\documentclass{amsart}
\usepackage{amssymb,amsfonts,amsmath,amsthm,amscd}
\usepackage[OT2,T1]{fontenc}
\usepackage{verbatim}
\usepackage{float}
\usepackage{mathtools}
\usepackage{cite}
\usepackage{csquotes}

\def\RR{\mathbb{R}}

\def\RH2{\RR H^2}
\def\gg{\mathfrak{g}}

\def\t{\texttt{T}}

\DeclareMathOperator{\diag}{diag}
\DeclareMathOperator{\aff}{aff}

\newtheorem{theorem}{{Theorem}}[section]
\newtheorem{corollary}{{Corollary}}[section]
\newtheorem{proposition}{{Proposition}}[section]
\theoremstyle{definition}
\newtheorem{remark}{{Remark}}[section]
\newtheorem{definition}{{Definition}}[section]

\newtheorem{examplex}{Example}[section]
\newenvironment{example}
  {\pushQED{\qed}\examplex}
  {\popQED\endexamplex}

\numberwithin{equation}{section}

\title{Geodesically Equivalent Metrics on Homogenous Spaces}
\author{Neda Bokan, Tijana Sukilovic, Srdjan Vukmirovic }
\date{24. December, 2017.}

\begin{document}
\maketitle

\abstract
Two metrics on a manifold are geodesically equivalent if sets of their unparameterized geodesics coincide. In this paper we show that if two left $G$-invariant metrics of arbitrary signature on homogenous space $G/H$ are geodesically equivalent, they are affinely equivalent, i.e. they have the same Levi-Civita connection. We also prove that existence of non-proportional, geodesically equivalent, $G$-invariant metrics on homogenous space $G/H$  implies that their holonomy algebra cannot be full. We give an algorithm for finding all left invariant metrics geodesically equivalent to a given left invariant metric on a Lie group. Using that algorithm we prove that no two left invariant metric, of any signature, on sphere $S^3$ are geodesically equivalent.
However, we present examples of Lie groups that admit geodesically equivalent, non-proportional, left-invariant metrics.
\endabstract



\section*{Introduction}

Let $(M^n,g)$ be a connected Riemannian  or
pseudo-Riemannian manifold of dimension $n\geq2$. We say that a
metric $\bar g$ on $M^n$ is \emph{geodesically equivalent} to $g$,
if every geodesic of $g$ is a reparameterized geodesic of $\bar
g$. We say that they are \emph{affinely equivalent}, if their
Levi-Civita connections coincide. We call a metric $g$
\emph{geodesically rigid}, if every metric $\bar g$, geodesically
equivalent to $g$, is proportional to $g$ (by the result of H.~Weyl  the coefficient of proportionality is a constant).
The geodesical equivalence is considered for metrics of arbitrary signature.
Moreover, the condition that $g$ and $\bar g$ are geodesically equivalent is linear in the sense that all metrics $\lambda g + \mu \bar g$, for constant $\lambda$ and $\mu$, are also geodesically equivalent to $g.$

Geodesically equivalent metrics are actively discussed
in the realm of general relativity
and have a close relation  with integrability and superintegrability (see \cite{Matveev2010} and references therein).
Also, there exists a connection between geodesical equivalence and holonomy groups (see for example \cite{Hall,Hall1,Wang2013}).

 The existence of geodesically equivalent metrics is quite restrictive. For example, if the complete metric $g$ is Einstain of non-constant sectional curvature and admits geodesically equivalent metric, then these two metrics are affinely equivalent \cite{MatveevRigid}. If the dimension of manifold is $3$ or $4$ the completeness condition can be dropped. By the results of N.~S.~Sinjukov \cite{Sinjukov1954} geodesically equivalent metrics on symmetric spaces are affinely equivalent.

Our research of geodesically equivalent metrics on Lie groups has been initiated by comprehensive discussions with Vladimir Matveev
during the visit to  Friedrich-Schiller-University Jena in 2013,  within DAAD cooperation program.
It came to our attention that the geodesically equivalent metrics have not been much considered in the frame of Lie groups.
In the Riemannian case, P.~Topalov \cite{topalov} considered the conditions under which the left invariant metric on a Lie group admits a non-trivial geodesically equivalent metric.
In our paper \cite{BSV} we classified left invariant metrics of Lorentz signature on 4-dimensional nilpotent Lie groups and noticed that if two of these metrics are geodesically equivalent they have to be affinely equivalent.
The  same is true for much wider class of metrics, namely for any two $G$-invariant metrics on homogenous space $G/H$. It is a very easy consequence of classical formulas which we prove in Theorem \ref{th:main}. Its strange that that simple fact is not stated or used in the  literature so far.

The structure of the paper is following.
First, in Section \ref{sec:prel}, we introduce the basic notation and recall some well known conditions for geodesical equivalence of metrics.

In Section \ref{sec:hom} we give a proof of   Theorem \ref{th:main} and draw some easy consequences.
In Theorem \ref{th:hol} we prove that existence of non-proportional geodesically equivalent metrics on homogenous space $G/H$ implies that the holonomy algebra cannot be full. Consequently, the curvature operator cannot have the maximal rank.

In Section \ref{sec:lie} we focus on geodesically equivalent left invariant metrics on Lie groups, that are special case of homogenous manifold.
In Proposition \ref{pr:method}  we prove an effective algorithm for finding the  set of all geodesically equivalent left invariant metrics to a given metric on a Lie group.
Using that method in Theorem \ref{th:s3} we show that three dimensional sphere $S^3$ as a Lie group does not have left invariant geodesically equivalent left invariant metric, of any signature.

In Example \ref{ex:g4} we show that there exists an indecomposable metric on a Lie group with affinely equivalent metrics (both left invariant). Finally, in Example \ref{ex:rh3} we show that on the same Lie group there exist both geodesically rigid and Riemannian metrics admitting affinely equivalent metrics.

\section{Preliminaries}
\label{sec:prel}

As it was already known to Levi-Civita \cite{Levi-Civita1896}, two
connections $\nabla=\{ \Gamma_{jk}^i\}$ and $\bar \nabla
=\{\bar{\Gamma}_{jk}^i\}$ have the same unparameterized geodesics,
if and only if their difference is pure trace: there exists a
covector $\phi_i$ such that
\begin{align}\label{eq:conn}
\bar\Gamma_{jk}^i=\Gamma_{jk}^i+\delta_k^i\phi_j+\delta_j^i\phi_k \ .
\end{align}
The reparametrization of the geodesics for $\nabla$ and $\bar
\nabla$ connected by \eqref{eq:conn} is done according to the
following rule: for a parameterized geodesic $\gamma(\tau)$ of
$\bar \nabla$, the curve $\gamma(\tau(t))$ is a parameterized
geodesic of $\nabla$, if and only if the parameter transformation
$\tau(t)$ satisfies the following ODE:
\begin{align*}
\phi_\alpha \dot{\gamma}^\alpha =
\frac{1}{2}\frac{d}{dt}\left(\log \left| \frac{d\tau}{dt}\right|
\right) \ .
\end{align*}
We denote by $\dot{\gamma}$ the velocity vector of $\gamma$ with
respect to the parameter $t$.

If $\nabla$ and $\bar \nabla$ related by \eqref{eq:conn} are
Levi-Civita connections of metrics $g$ and $\bar g$, respectively,
then one can find explicitly (following Levi-Civita
\cite{Levi-Civita1896}) a function $\phi$ on the manifold such
that its differential $\phi_{,i}$ coincides with the covector
$\phi_i.$ Indeed, contracting \eqref{eq:conn} with respect to $i$
and $j$, we obtain $\bar \Gamma _{\alpha i}^\alpha=\Gamma_{\alpha
i}^\alpha + (n+1)\phi_i$. From the other side, for the Levi-Civita
connection $\nabla$ of a metric $g$ we have $\Gamma_{\alpha
k}^\alpha = \frac{1}{2}\frac{\partial \log \left| \det (g)\right|
}{\partial x_k}.$ Thus
\begin{align}\label{eq:pari}
\phi_i = \frac{1}{2(n+1)}\frac{\partial}{\partial x_i}\log \left| \frac{\det (\bar g)}{\det (g)}\right|= \phi_{,i}
\end{align}
for the function $\phi:M\longrightarrow \RR$ given by
\begin{align}\label{eq:par}
\phi: = \frac{1}{2(n+1)}\log \left| \frac{\det(\bar g)}{\det (g)}\right| \ .
\end{align}
In particular, the derivative of $\phi_i$ is symmetric: $\phi_{i,j}=\phi_{j,i}$.

There is another convenient way to consider geodesically equivalent metrics extensively used by many authors (see for example \cite{Matveev2009,MatveevRigid,Sinjukov1954}), that we do not use in this paper.
Namely, if $(1,1)$-tensor $A=A(g, \bar g)$ is defined in the following way:
\begin{align*}
  A^i_j(g, \bar g) = \left| \frac{\det (\bar g)}{\det (g)}\right|^\frac{1}{n+1}\bar{g}^{ik}g_{kj},
\end{align*}
then the condition of geodesical equivalence of metrics $g$ and $\bar g$ can be formulated in terms of the linear system of PDEs on the components of the tensor $A$.

\section{$G$-invariant metrics on homogenous space $G/H$}
\label{sec:hom}

\begin{theorem}\label{th:main}
If two $G$-invariant metrics on a homogenous space $G/H$ are geodesically equivalent then they are affinely equivalent.
\end{theorem}
\begin{proof}
Let  $g$ and $\bar g$ be geodesically equivalent $G$-invariant metrics on $G/H .$ Since $G$ acts transitively and isometrically on  $G/H$ the  function $\phi$ from \eqref{eq:par} is constant. Therefore, components  $\phi_k$ of its covariant derivative are identically equal to zero. From formula \eqref{eq:conn} we have $\bar\Gamma_{jk}^i=\Gamma_{jk}^i,$ that is metrics $g$ and $\bar g$ have the same Levi-Civita connection so they are affinely equivalent.
\end{proof}

\begin{definition}
We say that a $G$-invariant metric $g$ on a homogenous space  $G/H$ is {\em invariantly rigid} if the only $G$-invariant metrics affinely equivalent to $g$ are metrics of the form $\lambda g,\, \lambda \neq 0$.
\end{definition}

\begin{proposition}\label{th:prop}
The indecomposable $G$-invariant  Riemannian metric on a homogenous space $G/H$ is invariantly rigid.
\end{proposition}

\begin{proof}
 If metrics $g$ and $\bar g$ are affinely equivalent, then $\bar g$ is parallel $(0,2)$ symmetric tensor with respect to the indecomposable Riemannian metric $g.$
 By the result of L.~P.~Eisenhart \cite{eisenhart},  $\bar g$ is proportional to $g$.
\end{proof}

\begin{example}\label{ex:sn}
Standard Riemannian metric $g_0$ on the $n$-dimensional sphere $S^n  \subset \RR ^{n+1}$ is a well known example of metric that admits \enquote{many} geodesically equivalent metrics. Namely, any matrix $A\in Sl(n, \RR)$  defines the mapping $f_A$
\begin{align}
f_A : S^n \to S^n, \quad \quad f_A (x):=\frac{Ax}{\vert Ax\vert }.
\end{align}
Since $f_A$ maps hyperplanes of $\RR ^{n+1}$ onto hyperplanes, it preserves great circles on $S^n.$ Then the pullback metric $g_A := f_A^* g_0$ is also a metric on $S^n$ that has the same geodesics as $g_0$, and therefore metrics $g_0$ and $g_A$ are geodesically equivalent. Moreover, $g_A \neq g_0$ if $A$ is not orthogonal matrix.

Note that  $g_0$ is Riemannian $SO(n+1)$ invariant metric on  $S^n = SO(n+1)/SO(n).$  From Proposition \ref{th:prop} it follows that it is invariantly rigid, i.e. non of the metrics $g_A \neq g_0$ is $SO(n+1)$ invariant. See Theorem \ref{th:s3} for more general statement regarding left invariant metrics on sphere $S^3$ viewed as a Lie group.
\end{example}

Existence of geodesically equivalent metrics has strong implications on holonomy of manifold, as next theorem shows.

 \begin{theorem}\label{th:hol}
   If $G$-invariant metric $g$ on a $G/H$ is not invariantly rigid, then it cannot have full holonomy algebra.
 \end{theorem}

 \begin{proof}
    If metric $\bar g$ is $G$-invariant  geodesically equivalent metric to $g$, from Theorem \ref{th:main} it follows that they are affinely equivalent.
    Thus, their connections coincide. Furthermore, they have the same holonomy algebra $hol(g) = hol(\bar g)\subseteq so(g)\cap so(\bar g)$.

    Suppose that holonomy is full and therefore $so(g)=so(\bar g)$ and fix a point $p \in G/H.$ Any     metric $g$ defines the conjugate operator $*:\, gl(T_pG/H)\rightarrow gl(T_pG/H)$ given by $g(A^*x, y)=g(x, Ay)$ for all $x, y\in T_pG/H$.
    Note that operator $*$ is reflection satisfying $*^2=id$ with $\pm 1$ eigenspaces respectively
    \begin{align*}
      \Lambda^-(g) &= so(g) =\{ X\,|\, SX+X^tS = 0\},\\
      \Lambda^+(g) & = sym(g)  =\{ X\,|\, SX-X^tS = 0\},
    \end{align*}
    where symmetric matrix $S$ represents metric $g$ at point $p$, in fixed some basis.
    Note that $*$ is isometry under the non-degenerate inner product on $gl(\gg)$ defined by $\langle A, B\rangle = tr(AB)$ and $gl(\gg) = \Lambda^-(g) \oplus\Lambda^+(g)$ is the orthogonal decomposition.
    Since $so(g)= so(\bar g)$ or, equivalently $ \Lambda^-(g)= \Lambda^-(\bar g)$, it follows that $ \Lambda^+(g)= \Lambda^+(\bar g)$ and therefore $*$ operator does not depend on metric. This implies that $\bar g=\lambda g$ for some $\lambda\neq 0$.
 \end{proof}
 The same statement holds for any pseudo-Riemannian manifold.
 Since curvature operators belong to holonomy algebra, this explains why the full holonomy and full rank of curvature operators do not appear in results regarding relations of holonomy and geodesical equivalence of metrics (see \cite{Hall1, Wang2013}).

\section{Case of left invariant metrics on a Lie group}
\label{sec:lie}

Any Lie group  $G$ with left invariant metric $g$ can be considered as a homogenous space. Therefore, from Theorem \ref{th:main} we have the following statement.

\begin{corollary}\label{th:cor}
If two  left invariant metrics on a Lie group $G$ are geodesically  equivalent they are affinely equivalent.
\end{corollary}

  In the Riemannian case, under certain conditions, this result has already been known to P.~Topalov (see \cite[Proposition 2]{topalov}). In the sequel we discuss geodesical equivalence of left invariant metrics on Lie groups.

The left invariant metric $g$ on $G$  is determined by its values on the Lie algebra. For the fixed basis  $\{e_1,\dots , e_n\}$ of the Lie algebra we denote by $S = (s_{ij})$ a symmetric and constant matrix of the metric, i.e. $s_{ij} = g(e_i, e_j) ,$ $i,j\in 1, \dots , n.$
The metric $g$  and hence the matrix $S$ can be of arbitrary signature.

Levi-Civita connection $\nabla$  of metric $g$ on $G$ is given in terms of connection 1-forms  $\omega = (\omega^i_j)$ by
  $\nabla e_i = \omega^j_i e_j$, or equivalently, $\nabla e^i = -\omega^i_j e^j.$

From the condition of compatibility with the metric, $\nabla g = 0$, we get that the matrix of connection is skew symmetric with respect to $g$
\begin{align}\label{eq:s}
   S\omega+\omega^\t S = 0.
\end{align}
On the other hand, the condition that the connection is torsion-free is equivalent to
\begin{align}\label{eq:de}
    e^i([e_j, e_k]) =  (\omega^i_m\wedge e^m)(e_j, e_k) = \omega ^i_k (e_j) - \omega ^i_j (e_k),
\end{align}
where $[\cdot , \cdot]$ denotes commutator in the Lie algebra.  Therefore, the unique Levi Civita connection matrix of 1-forms $\omega , $ of the metric $g,$ is solution of  \eqref{eq:s} and \eqref{eq:de}.

\begin{proposition}\label{pr:method}
Let $g$ be a left invariant metric on a Lie group $G$ represented by a constant symmetric matrix $S$ in a left invariant basis $\{e_1,\dots e_n\}$  and $\omega$ its Levi Civita connection matrix of  1-forms. Left invariant metric $\bar g$ is  geodesically equivalent to $g$ if and only if its matrix $\bar S$ in basis $\{e_1,\dots e_n\}$  belongs to the subspace
\begin{align}\label{eq:sbar}
    \aff (S):= \{ \bar S \enskip \vert \enskip \bar S\omega+\omega^\t \bar S = 0 \}.
\end{align}
\end{proposition}
\begin{proof}

 If $\bar S \in \aff (S)$ then $\bar g$ is parallel with respect to $\omega$ and \eqref{eq:de} holds. Because of the uniqueness of Levi Civita connection we have $\bar \omega = \omega .$ The metric $\bar g$ is affinely and therefore geodesically equivalent to $g.$

  Conversly, if  $\bar g$ is geodesically equivalent to $g,$ according to Corollary \ref{th:cor} they are affinely equivalent, i.e.  $\bar \omega = \omega.$ Therefore, $\bar g$ must be parallel with respect to $\omega$ so $\bar  S \in \aff (S).$
\end{proof}

This simple proposition gives a method for finding all left invariant metrics on a Lie group, that are geodesically equivalent to a given left invariant metric. We use this fact to prove the following theorem and also in the Example \ref{ex:g4}.
\begin{theorem} \label{th:s3}
All left invariant metrics on three dimensional sphere $S^3$  are invariantly rigid.
\end{theorem}
\begin{proof}
Sphere $S^3$ is isomorphic to multiplicative group of unit quaternions and its Lie algebra  is isomorphic to $so (3)$ with nonzero brackets
\begin{align}
\label{eq:so3}
[e_1, e_2] = e_3, \quad  [e_2, e_3] = e_1, \quad [e_3, e_1] = e_2.
\end{align}
The automorphism group of the Lie  algebra $so(3)$ is the group $SO(3).$ Any symmetric $3\times  3$ matrix $g$ can be diagonalized up to an automorphism $F\in SO(3)$: $S \sim F^\t SF = \diag (\alpha _1, \alpha _2, \alpha _3) = S_\alpha$.
This means that, up to the automorphism of the Lie algebra $so(3)$, any left invariant metrics $g$ on the Lie group $S^3$, in the basis $\{e_1, e_2, e_3\}$,  is represented by a diagonal matrix $S_\alpha, \alpha = (\alpha _1, \alpha _2, \alpha _3),\ \alpha _i\neq 0.$
Note that the signature of metric depends on signs of $\alpha _i.$

To find all left invariant metrics geodesically equivalent to metric $g$ represented by $S_\alpha$, we use the method from Proposition \ref{pr:method}.
By straightforward calculation, using \eqref{eq:s} and \eqref{eq:de}, we obtain the nonzero connection 1-forms $\omega = (\omega _i^j)$ of the metric:
\begin{align}\label{eq:omso3}
\omega _i ^j = \frac{1}{2}\left( 1 + \frac{\alpha _j}{\alpha _i} - \frac{\alpha _j}{\alpha _k} \right)e^k , \quad \omega _j ^i = -\frac{\alpha _i}{\alpha _j} \omega _i ^j.
\end{align}
Here $(i,j,k)$ is any cyclic permutation of indices $(1,2,3)$ (note that brackets \eqref{eq:so3} also have the same symmetries) and Einstein summation convention is not used.

If left invariant metric $\bar g$ on $S^3$  represented by $3\times 3$ symmetric matrix $\bar S$ in the basis $\{e_1, e_2, e_3\}$ is geodesically equivalent to $g$, it is described by \eqref{eq:sbar} for $\omega$ given by \eqref{eq:omso3} and $S = S_\alpha.$
Note that, since $\omega$ is matrix of 1-forms, the relations in \eqref{eq:sbar} are three matrix equations and  one can show that the only solution is $\bar S = \lambda S_\alpha, \, \lambda \neq 0 ,$ i.e. the subspace $\aff (S_\alpha)$ is one-dimensional and the metric $g$ is invariantly rigid.
\end{proof}

Note that metrics from previos theorem  include the metric of constant sectional curvature ($\alpha _1 = \alpha _2 = \alpha _3$), and the metrics of Berger spheres ($\alpha _1 = \alpha _2$).
However, the metric of constant sectional curvature on sphere $S^3$ is not rigid, since there exist many metrics geodesically equivalent to it (see Example \ref{ex:sn}), but they are not left invariant.

According to Proposition \ref{th:prop} and Theorem \ref{th:s3} invariantly rigid metrics are quite common,  but there are metrics on Lie groups, discussed in the sequel,  that are not invariantly rigid.

\begin{remark}
If on a pseudo-Riemannian  manifold exists one parallel vector field  of the  form  $v = v^k e_k$ for the metric $g= g_{ij}e^i\otimes e^j$, the family of affinely equivalent metrics is then given by
\begin{align}\label{eq:aff}
\bar g = \lambda \, g  + \mu \, v^*\otimes v^*= \lambda \, \left( g_{kj} \,
e^k \otimes e^j \right)  + \mu \, \left( v_k v_j \,
e^k \otimes e^j \right),\ \lambda, \mu \in \RR,
\end{align}
where $v^*$ denotes a dual form of $v$ with respect to the metric $g$.
If there exists more than one independent parallel vector fields the family of affinely equivalent metrics can be described in a similar way (for example, see \cite[Section 5]{Wang2013}).

Note that the orthogonal subspace  $v^\perp$ is also parallel. Therefore, if $v$ is not null, we have two complementary parallel distributions spanned by $v$ and $v^\perp$ and the manifold locally splits into the product of two manifolds in the sense of de Rham. This case is not interesting since for product manifolds we can always independently scale the metrics on each factor and obtain affinely equivalent metrics.
On the other hand, if $v$ is null, then $v\subset v^\perp$. The manifold does not have to split, but we still obtain nonproportional, affinely equivalent metric.

In case of a Lie group $G,$ if metric $g$ is left invariant, and $v$ is parallel left invariant vector field, then metric \eqref{eq:aff} is also left invariant and therefore metric $g$ on $G$ is not invariantly rigid.
\end{remark}

\begin{example}
\label{ex:g4}
  Let us consider the 4-dimensional 3-step nilpotent Lie group $G_4$. The corresponding Lie algebra is given by the non-zero commutators
  \begin{align*}
  [e_1, e_2]=e_3,\quad [e_1, e_3]=e_4.
  \end{align*}
  From the classification of Lorentz metrics in \cite{BSV} we choose the indecomposable metric
        \begin{align*}
       (g_{ij})= \left(
       \begin{array}{cccc}
         0 & 0 & 0 & 1 \\
         0 & 1 & 0 & 0 \\
         0 & 0 & \alpha & 0 \\
         1 & 0 & 0 & 0 \\
       \end{array}
     \right),\quad\alpha>0.
    \end{align*}
    The vector $v = e_4$ is parallel and null. Its dual vector is $v^* = (g_{ij}v^j)e_i = e^1$. Therefore, the family of affinely equivalent metrics is
    \begin{align*}
      \bar g =  \lambda \, g  + \mu e^1\otimes e^1,
    \end{align*}
    and the given metric on the Lie group $G_4$ is not invariantly rigid. Using the method described in Proposition \ref{pr:method}
    it can be shown that these are all left invariant metrics geodesically equivalent to $g$.
    Note that vector $v$ is also parallel with respect to metric $\bar g$, since the Levi-Civita connections coincide, and it is null with respect to $\bar g$.
\end{example}

\begin{example}
\label{ex:rh3}
This example shows that on the same Lie group one can have both affinely equivalent and invariantly rigid left invariant metrics.
  Consider the Lie group $G = \RR H_3\times\RR ^+$ where $\RR H_3$ denotes the real hyperbolic space. In the basis $\{e_1, e_2, e_3, e_4\}$, the non-trivial Lie brackets of the corresponding Lie algebra are
  \begin{align*}
  [e_1, e_3] = e_1,\quad  [e_2, e_3]=e_2.
  \end{align*}
  The family of non-equivalent left invariant Riemannian metrics is given by
  \begin{align*}
   S = \left(
       \begin{array}{cccc}
         1 & 0 & 0 & \beta \\
         0 & 1 & 0 & 0 \\
         0 & 0 & \alpha & 0 \\
         \beta & 0 & 0 & 1 \\
       \end{array}
     \right),\quad\alpha>0,\ \beta\geq 0.
  \end{align*}

  For $\beta=0$ the vector $e_4$ is parallel, but since the metric is Riemannian, the metric is decomposable and therefore is not invariantly rigid.

  If $\beta>0$, the metric is invariantly rigid. Also, the metric is indecomposable and without a parallel vector field.
\end{example}

 From the previous discussion  the question of describing all invariantly rigid, left invariant metrics on Lie groups arises naturally. Also, we do not know an example of $G$-invariant metric or left invariant metric that has more than one affinely equivalent metric. Classification of homogenous spaces $G/H$ that admit more than one such metric is still an open problem.

\vspace{1em}

\thanks{
   The authors would like to thank anonymous referee for pointing out that statements hold not only in case of Lie groups, but also in case of homogenous spaces.
The research of the second and the third authors are partially
supported within the project 174012 through the
Ministry of Education, Science and Technological Development of the Republic of Serbia.
}

\small{

}

{\small
{\em Authors' addresses}:

{\em Neda Bokan}, Faculty of Mathematics, University of Belgrade, Belgrade, Serbia
 e-mail: \texttt{neda@\allowbreak matf.bg.ac.rs}.

{\em Tijana \v{S}ukilovi\'{c}}, Faculty of Mathematics, University of Belgrade, Belgrade, Serbia
 e-mail: \texttt{tijana@\allowbreak matf.bg.ac.rs}.

{\em Sr\dj{}an Vukmirovi\'{c}}, Faculty of Mathematics, University of Belgrade, Belgrade, Serbia
 e-mail: \texttt{vsrdjan@\allowbreak matf.bg.ac.rs}.
}
\end{document}